\documentclass[11pt]{article}
\usepackage[english]{babel}
\usepackage{amsmath,amssymb,amsthm,latexsym}
\usepackage{fullpage}
\newtheorem{corollary}{Corollary}[section]
\newtheorem{theorem}[corollary]{Theorem}
\newtheorem{proposition}[corollary]{Proposition}
\newtheorem{lemma}[corollary]{Lemma}

\theoremstyle{definition}
\newtheorem{definition}[corollary]{Definition}
\newtheorem{remark}[corollary]{Remark}

\newcounter{abccnt}
\newenvironment{abc}{\begin{list}{\bf(\roman{abccnt})}{\usecounter{abccnt}
\labelwidth4ex \labelsep1ex \leftmargin6ex
\parsep3pt \itemsep1pt \topsep3pt}}{\end{list}}

\newcommand{\abs}[1]{\left\vert#1\right\vert}
\newcommand{\set}[1]{\left\{#1\right\}}

\newcommand{\complex}{\mathbb C}
\newcommand{\zint}{\mathbb Z}

\newcommand{\To}{\longrightarrow}
\newcommand{\upontop}[2]{\genfrac{}{}{0pt}{}{#1}{#2}}

\begin{document}

\title{Characteristics of Invariant Weights\\ 
  Related to Code Equivalence over Rings}

\author{Marcus Greferath, Cathy Mc Fadden, and Jens Zumbr\"agel%
  \thanks{School of Mathematical Sciences, University College Dublin,
    and Claude Shannon Institute for Discrete Mathematics, Coding,
    Cryptography, and Information Security, Dublin, Republic of
    Ireland.  Email: \{marcus.greferath, cathy.mcfadden,
    jens.zumbragel\}@ucd.ie\newline This work was supported in part by
    the Science Foundation Ireland under Grants 06/MI/006 and
    08/IN.1/I1950.}}

\date{}
 

\maketitle

\begin{abstract}
  \noindent
  The Equivalence Theorem states that, for a given weight on the
  alphabet, every linear isometry between linear codes extends to a
  monomial transformation of the entire space. This theorem has been
  proved for several weights and alphabets, including the original
  MacWilliams' Equivalence Theorem for the Hamming weight on codes
  over finite fields. The question remains: What conditions must a
  weight satisfy so that the Extension Theorem will hold? In this
  paper we provide an algebraic framework for determining such
  conditions, generalising the approach taken in \cite{greferath06}.\medskip

  \smallskip
    
\end{abstract}


\section*{Introduction}

Two linear codes of the same length over a given alphabet are said to
be equivalent if there exists a (weight preserving) monomial
transformation mapping one to the other. MacWilliams in her doctoral
thesis \cite{macwilliams63} proved that when the alphabet is a finite
field any linear Hamming isometry between linear codes will extend to
a monomial transformation. Thus the equivalence question can be seen
as an extension problem. A character theoretic proof of this Extension
Theorem in \cite{ward96} led to a generalisation of this theorem for
codes over finite Frobenius rings in \cite{wood97}. Indeed in
\cite{wood08} it was shown that linear Hamming isometries extend
precisely when the ring is Frobenius.

In the seminal paper on ring linear coding \cite{hammons94} it was
noticed that weights other than the Hamming weight would play a
significant role, such as the Lee weight over $\zint_{4}$.  The
concept of a homogenous weight was first introduced in
\cite{constantinescu96a} where a combinatorial proof of the Extension
Theorem for this weight and codes over $\zint_{m}$ is provided. In
\cite{greferath00} we see that every homogeneous isometry is a Hamming
isometry yielding the Extension Theorem for the homogeneous weight and
codes over finite Frobenius rings.  This paper followed the
combinatorial tack of \cite{constantinescu96a} for the $\zint_{m}$
case. For the more general case of codes over modules the Extension
Theorem holds for Hamming weights as seen in \cite{greferath04}.

Following from the chain ring result of \cite{greferath05}, obtained
by examining the generation of invariant weights, a complete
characterisation of those weights for which the Equivalence Theorem
holds for codes over $\zint_{m}$ is supplied in
\cite{greferath06}. Here we extend the ideas of that paper to more
general rings, outlining a strategy for attaining necessary and
sufficient conditions for a weight to satisfy the Extension Theorem.

We begin in Section~\ref{sec:prelims} by revising some key properties
of the M\"{o}bius Function and chain rings.  In
Section~\ref{sec:weightdefs} we define codes, weights and the
equivalence condition for the ring case.  In
Section~\ref{sec:convcorr} we describe the structural context so
crucial to the elegance and seeming simplicity of our results.  Then,
after a short section on finite products of chain rings, we finally
provide in Section~\ref{sec:modgen} a sufficient condition for an
invariant weight to satisfy the generalised MacWilliams' Equivalence
Theorem.


\section{Algebraic and Combinatorial Preliminaries}%
\label{sec:prelims}

In the following sections we will harness the power of M\"{o}bius
Inversion to prove our most vital results. We state the key points
here, for more details see \cite{roman92}.

\begin{definition}\label{mobiusdefn}
  Consider a field ${\mathbb F}$ and a finite partially ordered set
  $P$ with partial ordering $\leq$. The {\em M\"{o}bius function\/},
  $\mu : P \times P \longrightarrow {\mathbb F}$, is defined by
  $\mu(x,y)=0$ for $x\nleqslant y$, and any of the four equivalent
  statements:
  \begin{abc}
  \item \quad $ \mu (x,x) = 1$ \quad and \quad $\displaystyle 
     \sum_{x \leq z \leq y}{\mu(z,y)} = 0$ \ for $x<y$
  \item \quad $ \mu (x,x) = 1$ \quad and \quad $\displaystyle
    \sum_{x \leq z \leq y}{\mu(x,z)} = 0$ \ for $x<y$
  \item \quad $ \mu (x,x) = 1$ \quad and \quad $\displaystyle
    \mu (x,y) = -\! \sum_{x < z \leq y}{\mu(z,y)} $ \ for $x<y$
  \item \quad $ \mu (x,x) = 1$ \quad and \quad $\displaystyle
    \mu (x,y) = -\! \sum_{x \leq z < y}{\mu(x,z)} $ \ for $x<y$
  \end{abc}
\end{definition}

\begin{theorem}
  Let $P$, ${\mathbb F}$, and $\mu$ be as above and let $f,g$ be
  functions from $P$ to ${\mathbb F}$.  If $P$ has least element 0
  then:
  \[ g(x) = \sum_{ y \leq x}{f(y)} \;\;\mbox{for all $x\in P$}\quad
  \Leftrightarrow \quad f(x) = \sum_{y \leq x}{g(y) \mu (y,x)} \;\;
  \mbox{for all $x\in P$}. \] 
  If additionally the partially ordered set $P$ has a greatest
  element~1 then:
  \[ g(x) = \sum_{ x \leq y}{f(y)} \;\;\mbox{for all $x\in P$}\quad 
  \Leftrightarrow \quad f(x) = \sum_{x \leq y}{g(y) \mu (x,y)}\;\;
  \mbox{for all $x\in P$}. \]
\end{theorem}

Now we include a brief summary of the key properties of chain rings
(c.f. \cite{lang02}, \cite{mcdonald74}, \cite{Kasch82}). In all of our
discussion let $R$ be a finite associative ring with identity
1. Denote by $R^{\times}$ the group of multiplicatively invertible
elements of $R$.

\begin{definition}
  A ring $R$ is called a {\em left chain ring\/} if the set of left
  ideals of $R$ forms a chain under the partial ordering of
  inclusion. Similarly for {\em right chain ring\/}. If $R$ is both a
  left and right chain ring then it is called a {\em chain ring\/}.
\end{definition}

The following theorem, combining Theorem 1.1 of \cite{nechaev73} and
Lemma 1 of \cite{clarke73}, demonstrates the numerous equivalent
definitions of a finite chain ring.  Recall a {\em principal
  left ideal ring\/} is a ring with identity in which each left ideal
is left principal, and a {\em principal ideal ring\/} is a ring which
is both a principal left ideal ring and a principal right ideal ring.

\begin{theorem}
  The following are equivalent:
  \begin{abc}
  \item $R$ is a local principal ideal ring.
  \item $R$ is a left chain ring.
  \item $R$ is a chain ring.
  \item $R$ is a local ring and ${\rm rad}(R)$ is a left principal
    ideal.
  \item Every one-sided ideal of $R$ is two-sided and belongs to the
    chain\\ $R \rhd {\rm rad}(R) \rhd ... \rhd {\rm rad}(R)^{n-1}
    \rhd {\rm rad}(R)^{n} = \{0\}$, for some $n \in {\mathbb N}$.
  \end{abc}
\end{theorem}

\begin{remark}
  Note that in the above if $n > 1$, then ${\rm rad}(R)^i = R \pi^i =
  \pi^i R$ for any $\pi \in {\rm rad}(R) \setminus {\rm rad}(R)^2$,
  $i \in \{1 \dots n\}$.
  Wood noted in \cite{wood00} that
  \[ {\rm rad}(R)^i \setminus {\rm rad}(R)^{i+1} = R^{\times} \pi^i 
    = \pi^i R^{\times} \:. \]
  This property extends in a natural way to finite direct products of
  chain rings and, combined with our structural approach, facilitates
  the proof of the main theorems herein.
\end{remark}


\section{Weight Functions and the Equivalence Theorem}%
\label{sec:weightdefs}

Let the left symmetry group of any function $f: R \to \complex$ be
given by ${\rm Sym}_\ell(f) := \{u \in R^{\times} \mid f(x) = f(ux)\
\forall x \in R \}$ and the right symmetry group by ${\rm Sym}_r(f) :=
\{u \in R^{\times} \mid f(xu) = f(x)\ \forall x \in R\}$.  By a weight
on $R$ we mean any function $w : R \to \complex$ satisfying $w(0) =
0$. A weight $w$ is called {\em invariant\/} if both symmetry groups
are maximal, i.e.\ if they coincide with $R^\times$. Note that for a
finite ring $Rx \, =\, Ry$ implies $R^{\times}x = R^{\times}y$, as
detailed in \cite{wood99b}, hence if ${\rm Sym}_\ell(w)= R^{\times}$
then $w(x) \, =\, w(y)$.

\begin{definition}
  An invariant weight $w$ on $R$ is called {\em homogeneous\/}, if
  there exists a real number $c\geq 0$ such that for all $x\in R$
  there holds:
  \[ \sum_{y\in Rx} w(y) \; =  \; c \,|Rx|\quad \mbox{ if } x\neq 0 \:. \]
\end{definition}

The concept of a homogeneous weight was originally introduced in
\cite{constantinescu96} and further generalised in two different
directions: one is given in the work by Nechaev and
Honold~\cite{nechaev99}, in which the term homogeneous weight is
reserved for those with constant average weight on {\em every\/}
nonzero ideal. The other can be seen in the work by Greferath and
Schmidt~\cite{greferath00}, where the constant average property is
postulated only for principal ideals. Both definitions are equivalent
for the class of finite Frobenius ring.

This article follows the line given in \cite{greferath00} and hence,
every finite ring allows for a homogeneous weight. The case of average
value $1$ is referred to as the {\em normalised\/} homogeneous weight
$w_{\rm hom}$.

\begin{definition}
  The normalised homogeneous weight $w_{\rm hom}: R \to \mathbb{R}$ is
  given by
  \[ w_{\rm hom} (x) = 1 - \frac{\mu(0, Rx)}{|R^\times x|} \:, \]
  where $\mu$ is the M\"{o}bius function on the lattice of principal
  ideals of $R$, from Definition \ref{mobiusdefn}, and $|R^\times x|$
  counts the number of generators of the ideal $Rx$ as proved in
  \cite{greferath00}.
\end{definition}

Given a positive integer $n$, any weight $w : R \to \complex$ shall be
extended to a function on $R^{n}$ by defining $w(x) := w(x_1) + w(x_2)
+ \dots + w(x_n)$ for $x \in R^{n}$. Suppose that $C$ is a linear code
of length $n$ over $R$, i.e.\ an $R$-submodule of ${}_RR^{n}$. A map
$\phi : C \to {}_RR^{n}$ is called a $w$-isometry if $w ( \phi(x) )
= w(x)$ for all $x \in C$.

A bijective module homomorphism $\phi : {}_RR^{n} \to {}_RR^{n}$
is called a {\em monomial transformation\/} if there exists a
permutation $\pi$ of $\{1 \dots n\}$ and units $u_1, \dots, u_n \in
R^{\times}$ such that $\phi(x) = ( x_{\pi(1)} u_1 , \dots , x_{\pi(n)}
u_n )$ for every $x =(x_1 , \dots , x_n ) \in R^{n}$. If all the units
$u_{i}$ are contained in a subgroup $G$ of $R^{\times}$ we call it a
{\em $G$-monomial transformation\/}.
 
Clearly any ${\rm Sym}_r(w)$-monomial transformation will be a
$w$-isometry for any weight $w$ and hence restricts to a $w$-isometry
on every linear code $C \subseteq {}_RR^{n}$. Conversely we may ask
if a given linear $w$-isometry $\phi : C \to {}_RR^{n}$, defined on
a linear subcode $C$ of $R^{n}$ is a restriction of an appropriate
monomial transformation of $R^{n}$. This is the essence of
MacWilliams' Equivalence Theorem:

\begin{theorem}[MacWilliams \cite{macwilliams63}]
  Every linear Hamming isometry between linear codes of the same
  length over a finite field can be extended to a monomial
  transformation of the ambient vector space.
\end{theorem}

\begin{definition}
  Suppose $w$ is an arbitrary weight. We say that {\em MacWilliams'
    Equivalence Theorem\/} (or the {\em Extension Theorem\/}) holds
  for $w$ on $R$ if for each positive integer $n$, linear code $C$ in
  ${}_RR^{n}$ and linear $w$-isometry $\phi : C \to {}_RR^{n}$
  there exists a ${\rm Sym}_r(w)$-monomial transformation of $R^{n}$
  which extends $\phi$.
\end{definition}


An obvious necessary condition for MacWilliams' Equivalence Theorem to
hold for a weight~$w$ on~$R$ is that all $w$-isometries are injective.


\section{Convolution and Correlation}%
\label{sec:convcorr}

Two key operations, convolution and correlation, allow us to define a
module of weights over an algebra of complex functions. Consider the
set $\complex^R$ of all functions $\{ f \mid f:R \to \complex\}$. For
$f, g \in \complex^R$ and for $\lambda \in \complex$ we define
addition and scalar multiplication by
\begin{align*} 
  (f + g)(x) &:= f(x) + g(x) \\
  (\lambda f)(x) &:= \lambda f(x) \:, 
\end{align*}
then $V = [\complex^R, +, 0; \complex]$ is a $\complex$-vector space.

\begin{definition}
  Let $f$ and $g$ be elements of $\complex^R$. We define the {\em
    multiplicative convolution\/} as a mapping:
  \begin{align*}
    \ast :  \complex^R \times \complex^R \To \complex^R, \ \ \ 
    (f,g) \mapsto f \ast g \\   
    \mbox{where }  \ \ ( f \ast g ) (x) := 
    \sum_{\upontop{a,b \in R,}{ ab=x}}{f(a) g(b)} \:.
  \end{align*}
\end{definition}

For each element $r\in R$ denote by $\delta_r$ the function defined
by:
\[ \delta_r(x) :=  \left\{
  \begin{array}{lcl}
    1 & \ :\  & x = r\\
    0 & : & \mbox{otherwise} \:. 
  \end{array}
\right. \]
We extend the notation to each subset $A$ of $R$ by defining $\delta_A
= \sum_{a \in A}{\delta_a}$. The multiplicative identity of the $\ast$
operation is $\delta_1$.

\begin{lemma}
  $\complex^R$, with addition and scalar multiplication as above and
  the operation $\ast$, is an algebra over $\complex$, which we call
  $\complex[R, \ast]$, or simply $\complex[R]$.
\end{lemma}

\begin{proof}
  
  It is clear that convolution is associative and additively
  distributive and that $\delta_{1}$ is indeed an identity.
  If $\lambda$ in $\complex$, then $\lambda(f \ast g) = (\lambda f) \ast
  g = f \ast (\lambda g).$ Thus $\complex[R]$ is indeed a complex
  algebra.
\end{proof}  

Note that $\delta_r \ast \delta_s = \delta_{r s}$ and that the set
$\set{ \delta_r \mid r \in R }$ forms a $\complex$-basis of
$\complex[R]$.

\begin{definition}
  Let $f, g$ and $w$ be elements of $\complex^R$. The left and right
  {\em multiplicative correlations\/} are given by
  \begin{align*}
    \circledast': S \times W &\To W \:, \qquad (f,w) \mapsto f \circledast' w\\
    \circledast\,: W \times S &\To W\:, \qquad (w,g) \mapsto w \circledast g
  \end{align*}
  respectively, where 
  \begin{align*}
    ( f \circledast' w ) (x) &\;:=\; \sum_{r \in R}{ f(r) w(xr) } \\
    ( w \circledast  g ) (x) &\;:=\; \sum_{r \in R}{ w(rx) g(r) } \:.
  \end{align*}
\end{definition}

\begin{lemma}
  Let $f,g,w \in \complex^R$, then convolution and correlation have
  the following relationships:
  \begin{align*}
    (f \ast g) \circledast' w &=  f \circledast' (g \circledast' w) \\
    w \circledast (f \ast g) &= ( w \circledast f) \circledast g \\
    g \circledast'( w \circledast f) &= (g \circledast' w) \circledast f \:.
  \end{align*}
\end{lemma}

\begin{lemma}
  The complex vector space $V = [\complex^R, +, 0; \complex]$ is a
  $\complex[R]$-bimodule under the left and right $\complex[R]$-ring
  multiplications
  \begin{align*}
    (f,w) &\To f \circledast' w \\
    (w,g) &\To w \circledast  g.
  \end{align*} 
\end{lemma}

\begin{proof}
  Combining additive distribution with the preceeding Lemma the result
  is evident.
\end{proof}

\begin{lemma}
  The set $\complex \delta_0 $ is a two-sided ideal in the algebra
  $\complex[R, \ast]$ where
  \[ \complex \delta_0  = \{ c \delta_0 \mid c \in \complex \} \:. \]
\end{lemma}

With this two-sided ideal we can immediately form the factor algebra
$\complex[R, \ast]/ \complex \delta_0 $ which we call $\complex_0[R]$.
\begin{definition}\rm
  We define the set $V_0$ to be those functions $w$ in $V$ which
  satisfy $w(0) = 0$.
  \[ V_0 := \{ w \in V \mid w(0) = 0 \} \:. \]
\end{definition}

As $w \circledast \delta_0 = 0$ for all $w \in V_0$ this induces a
natural right action of $\complex_0[R]$ on $V_0$ by
\[  w \circledast (f + \complex \delta_0 )  := w \circledast f \:, \]
where $g = f + \complex \delta_0 $ is any element of
$\complex_0[R]$ and $w \in V_0$. Similarly there exists a left
action via~$\circledast'$.


\section{Direct Product of Chain Rings}%
\label{sec:dirprod}

From now on let the ring $R$ be a finite product of finite chain rings
$R_i$, say $R = R_1 \times R_2 \times \dots \times R_r$, with Jacobson
radicals generated by $p_1, p_2, \dots, p_r$ of nilpotency $d_1, d_2,
\dots, d_r$ respectively.  We view elements of $R$ as $r$-tuples of
chain ring elements i.e.\ $a\in R$ represented as $a = (a_1, a_2,
\dots, a_r)$ where each $a_i\in R_i$.  Operations, including
multiplication, are performed component-wise.  The set of generators
of the ideals of $R$ is given by $\{R^{\times}e \mid e \in E \}$ where
$E$ are the representatives
\[ E = \{ {p_1}^{e_1} {p_2}^{e_2} \dots {p_r}^{e_r} = 
e \mid 0 \leq e_i \leq d_i \} \:. \]
The lattice of principal left ideals of $R$ may be described by
$E(_RR) = {\{ Re \mid e \in E \}}$.

We have for $e = p_1^{e_1}\dots p_r^{e_r}, f = p_1^{f_1}\dots
p_r^{f_r}\in E$ the relations $e_i\le f_i\ \forall i$ if and only if
$Re\ge Rf$, and in this case we write $e\ge f$.  The {\em socle} of
any $R$-module ${}_RM$ is the sum of the minimal submodules of
${}_RM$. When ${}_RM$ is the ring as a left module over itself this is
the sum of the minimal left ideals.  Here the representative of the
socle is $s = {p_1}^{d_1-1}{p_2}^{d_2-1} \dots{p_r}^{d_r-1}$ by the
nature of the direct product.

Let us take a look at how the M\"{o}bius function behaves on the
partially ordered set of principal ideals. For a chain ring $T$, where
$\pi$ generates the radical with nilpotency index $h$, the function is
described by:
\[ \mu(T \pi ^x, T \pi ^y) =  \left\{
  \begin{array}{lcl}
    1  & \ :\  &  x = y \\
    -1  & : &   x = y + 1 \\
    0 & : &  x > y + 1 \:. 
  \end{array}
\right. \]
Translation invariance in the lattice of principal ideals of a direct
product of finite chain rings means we are only interested in the
values of the M\"{o}bius function takes within the socle. The nature
of the lattice, combined with binomial theorem arguments, allows us to
determine those values we will be interested in.

\begin{lemma}\label{mobiuslem}
  The M\"{o}bius function takes values for all $e = p_1^{e_1}\dots
  p_r^{e_r}\in E$,
  \[ \mu(0,Re) =  \left\{
    \begin{array}{lcl}
      (-1)^{\Sigma{(d_{i}- e_{i})}}  & \ :\  & Re \leq \rm{Soc}(R)\\
      0 & : &  Re \nleqslant \rm{Soc}(R) \:.
    \end{array}
  \right. \]
\end{lemma}


\section{MacWilliams' Extension Theorem by Module Generation}%
\label{sec:modgen}

For any functions $f, g \in \complex[R]$ note that ${\rm Sym}_\ell(f
\ast g) \supseteq {\rm Sym}_\ell(f)$ and in a similar line we have
${\rm Sym}_r(f \ast g) \supseteq {\rm Sym}_r(g)$.

\begin{lemma}
  Symmetry groups are inherited as follows for correlation
  \begin{align*}
    {\rm Sym}_\ell (w \circledast g) &\supseteq {\rm Sym}_r (g) \\
    {\rm Sym}_r (f \circledast ' w) &\supseteq {\rm Sym}_\ell (f) \:.
  \end{align*}
\end{lemma}  

\begin{lemma}
  Define $S = \{ f \in \complex_0[R] \mid f(x u) = f(x)\ \forall
  x \in R, u \in R^{\times} \}$ and let the invariant weights from
  Section~\ref{sec:weightdefs} be denoted by $W = \{ w \in V_0 \mid
  {\rm Sym}_\ell (w) = R^{\times} = {\rm Sym}_r (w) \}$.  Then $W$
  is a right $S$-module under correlation $\circledast$ in a naturally
  inherited way.
\end{lemma}

We illustrate this by considering the correlation $w \circledast f$ at
$ux$ and $xu$.
\begin{align*}
  w \circledast f (ux) &= \sum_{r \in R}{w (rux) f(r)} \\
  &= \sum_{s \in R}{w (sx) f(s u^{-1})}
\end{align*}
When $f$ is right invariant this will be simply $w \circledast f
(x)$. Now \[ w \circledast f (xu) = \sum_{r \in R}{ w(rxu)f(r)} \]
which will be $w \circledast f (x)$ when $w$ is right invariant. Hence
the correlation is in $W$ when $w \in W$ and $f \in S$.
 
We re-examine the Extension Theorem with this new perspective. We aim
to classify all weights that generate $W$ as a right $S$-module. This
will then yield MacWilliams' Equivalence Theorem for these weights due
to the following results, equivalent to those in \cite{greferath05}.

\begin{lemma}\label{lemiso}
  If $\phi$ is a $w$-isometry then $\phi$ is a $(w \circledast
  s)$-isometry for all $s \in S$.
\end{lemma}

\begin{proof} 
  Let $\phi$ be a $w$-isometry. Then $w(\phi(x)) = w(x)$. Examine $(w
  \circledast s)(\phi(x))$:
  \begin{align*}
    w \circledast s (\phi(x)) &= \sum_{r \in R}{w (r\phi(x)) s(r)} \\
    &= \sum_{r \in R}{w (\phi(rx)) s(r)} \\
    &= \sum_{r \in R}{w (rx) s(r)} 
  \end{align*}
  which is $w \circledast s (x)$ and thus $\phi$ is a $(w \circledast
  s)$-isometry.
\end{proof}

\begin{remark}
  Let $R$ be a Frobenius ring.  If $w \circledast S = W$ then $w
  \circledast h = w_{H}$ for some $h \in S$ where $w_{H}$ denotes the
  Hamming weight. Since every $w$-isometry is a $(w \circledast
  h)$-isometry, by Lemma~\ref{lemiso}, we have that MacWilliams'
  Extension Theorem holds for $w$.
\end{remark}

Continuing with our notation for $R$ as before we define the natural
basis for~$S$.
\begin{definition}
  For each $e \in E$ define the basis element
  \[ \varepsilon_e := 
  \ \frac{1}{\abs{R^{\times}e}} \sum_{a \in R^{\times}e}{\delta_{a}} = 
  \ \frac{1}{\abs{R^{\times}e}} \ \delta_{R^{\times}e} \:. \]
\end{definition}

By abuse of notation for all $e \in E$ we denote by $e^{\bot}$ the
orthogonal ideal to $Re$, namely $e^{\bot} = (Re)^{\bot} = \set{r \in
  R \mid rs = 0\ \forall s \in Re}$. Note that when $e = p_1^{e_1}\dots
p_r^{e_r}$ and $e^{\bot} = p_1^{e_1^{\bot}}\dots p_r^{e_r^{\bot}}$
then $e_i^{\bot} = d_i-e_i$.
We define the set $\{ \eta_{x} \mid x\in E\setminus\{0\} \}$ in $S$
by \[ \eta_x := \sum_{x^{\bot} \le t } {\mu (0,Rxt) \ \varepsilon_t } \:, \]
where $\mu$ is the M\"{o}bius function induced by the lattice of left
principal ideals under the partial order of inclusion. The values are
as given in Lemma \ref{mobiuslem}.  Since $\mu(0, Rz)=0$ for $Rz
\nleqslant Soc(R)$ we need only include those $z = xt$ with indices
$z_i = d_i$ or $z_i = d_i -1$ in the sum.

\begin{lemma}
  The set $\{ \eta_x \mid x\in E\setminus\{0\} \}$ as defined above
  is a basis of $S$.
\end{lemma}

\begin{proof}
  Define the indicator function
  \[ \mathbf{1}_{a \leq b} =  \left\{
    \begin{array}{lcl}
      1  & \ :\  & Ra \leq Rb\\
      0 & : &  Ra \nleqslant Rb \:.
    \end{array}
  \right. \] The matrix $(\mathbf{1}_{a \leq b} )_{a, b \in E}$ will
  be upper triangular and invertible with respect to the usual rank
  ordering.  Thus the matrix given by \[ \left( \mu (0, Ra^{\bot} b)
    \mathbf{1}_{a \leq b} \right)_{a,b \in E} \] will be invertible if
  and only if $\mu (0, Ra^{\bot} b)$ is nonzero when $b=a$. By
  applying the permutation $a \mapsto a^{\bot}$ we acquire an
  equivalent statement: \[ \left( \mu (0, Ra b) \mathbf{1}_{a^{\bot}
      \leq b} \right)_{a,b \in E} \mbox{ is invertible } \quad
  \Leftrightarrow \quad \mu (0, Ra a^{\bot}) \neq 0 \] which is true
  by orthogonality. As this matrix describing the transform from
  $\set{\varepsilon_{e}}$ to $\set{\eta_{x}}$ is invertible it is
  clear that the $\set{\eta_{x}}$ form a basis of $S$.
 \end{proof}

We examine the action of correlation on this basis of $S$. 

\begin{proposition}
  \[ ( w \circledast \eta_{x})(y) =  \left\{
    \begin{array}{lcl}
      \sum_{x^{\bot} \leq t}{\mu(0,Rxt) w(ty)} & : & Rx \leq Ry\\
      0 & \;\;:\;\; & Rx \nleqslant Ry \:.
    \end{array}
  \right. \]
\end{proposition}

\begin{proof}
  First we expand the correlation to the formula above (with which we
  will examine the case when $Rx \nleqslant Ry$).
  \begin{align*}
    w \circledast \eta_x (y) &= \sum_{r \in R}{ w(ry) \eta_x (r) } \\
    &= \sum_{r \in R}{ w(ry) \sum_{x^{\bot} \leq t }{\mu (0, Rxt) \
        \varepsilon_t (r) }}\\
    &= \sum_{x^{\bot} \leq t}{ \mu(0, Rxt) \sum_{r \in R}{ w(ry) \
        \varepsilon_t (r) }}
  \end{align*}
  The second sum will be nonzero only for those $r$ in $R^{\times}t$
  each of which will contribute $ \frac{1}{\abs{R^{\times} t}} w (t
  y)$. This yields the desired description.

  We use the notation $x = p_1^{x_1}\dots p_r^{x_r}$, $y =
  p_1^{y_1}\dots p_r^{y_r}$, $t = p_1^{t_1}\dots p_r^{t_r}$,
  etc. Suppose $Rx \nleqslant Ry$, so there exists a $k$ such that
  $y_{k} > x_{k}$ and it follows $x_{k} < d_{k}$. Take $t$ in the sum
  above, $Rx^{\bot} \leq Rt$ and $Rxt \leq $ Soc$(R)$. The
  implications are $d_{i} - x_{i} \geq t_{i}$ and $x_{i} + t_{i} \geq
  d_{i} -1$. Hence either $t_{k} = d_{k} -x_{k}$ or $t_{k} = d_{k} -
  x_{k} -1$. This implies $(ty)_{k} = t_{k} + y_{k} \geq d_{k}$ since
  $y_{k} - x_{k} -1 \geq 0$. Thus $w(ty)$ will be the same for either
  option of $t_{k}$, namely $w(tp_{k} y) = w(ty)$ when $t_{k} = d_{k}
  - x_{k} -1$.
  We divide the sum into two parts, splitting over the value of
  $t_{k}$.
  \begin{align*}
    \sum_{x^{\bot} \leq t}{ \mu (0, Rxt) w (t y)} 
    &= \sum_{\upontop{x^{\bot} \leq t}{t_{k}=d_{k} - x_{k} }}{ \mu (0, Rxt) w (t y)} 
    + \sum_{\upontop{x^{\bot} \leq t}{t_{k}=d_{k} -x_{k} -1}}{ \mu (0, Rxt) w (t y)} \\
    &= \sum_{\upontop{x^{\bot} \leq t}{t_{k}=d_{k} - x_{k} -1}}{ \mu (0, Rxtp_{k}) 
      w (t p_{k} y) + \mu (0, Rxt) w(t y)} \\
    &=  \sum_{\upontop{x^{\bot} \leq t}{t_{k}=d_{k}-x_{k}-1}}{ (\mu (0, Rxtp_{k}) 
      + \mu (0, Rxt) ) w(t y)} 
  \end{align*}
  Now since $\mu (0, Rxt) = (-1)^{\sum{d_j - x_j -t_j}}$, it follows
  that $\mu (0, Rxt) + \mu (0, Rxtp_k) = 0$. Hence $w \circledast
  \eta_x (y) = 0$ when $Ry \nleqslant Rx$.
\end{proof}

Thus the matrix of coefficients of the weight $w$ with respect to the
basis $\{ \eta_x \mid x \in E\setminus\{0\}\}$ is triangular. We
require for $w$ to generate $W$ that the diagonal elements are
nonzero, indeed this is sufficient. Combining all of these elements we
arrive at our main theorem.

\begin{theorem}
  Let $R$ be a finite direct product of finite chain rings with $E$
  the set of representatives of the ideals of $R$. If $w \in W$ with
  \[ \sum_{x^{\bot} \leq t}{\mu(0, Rxt) w(tx)} \neq 0 \quad
  \mbox{ for all } x \in E\setminus\{0\} \:, \]
  then MacWilliams' Equivalence Theorem holds for $w$.
\end{theorem}

We remark that a finite commutative ring is a direct product of chain
rings if and only if it is a principal ideal ring.  Hence the theorem
applies in particular to finite commutative principal ideal rings.


\section*{Conclusion}

By considering the module of invariant weights in terms of an algebra
of complex functions we have determined the conditions an invariant
weight defined on a direct product of chain rings must satisfy for
MacWilliams' equivalence theorem to hold. Thus provided these
conditions are satisfied all isometries of that weight will extend to
monomial transformations.



\bibliographystyle{plain}
\bibliography{biblio1}

\end{document}